\documentclass[11pt,a4paper]{article}

\usepackage[T1]{fontenc}
\usepackage{amsmath, MnSymbol, amsthm, inconsolata, parskip, xcolor, graphicx, 
    array, enumitem, geometry, titlesec, blkarray, url, datetime, tikz, multirow, 
    etoolbox, dsfont, xfrac, caption, hyperref}

\usetikzlibrary{positioning,arrows,calc,shapes.geometric,decorations.pathmorphing}

\hypersetup{
	pdfstartview=FitH,
	hidelinks,
	colorlinks,
    urlcolor=green!60!black,
	bookmarks,
	bookmarksdepth = 2,
	bookmarksopen = true,
	pdftitle={On the Walks and Bipartite Double Coverings of Graphs with the same Main Eigenspace},
	pdfauthor={Luke Collins and Irene Sciriha},
	pdfsubject={Mathematics},
	pdfcreator={University of Malta},
	pdfdisplaydoctitle=true,
	pageanchor=false}

\usepackage[nameinlink, noabbrev]{cleveref}

\DeclareFontFamily{U}{BOONDOX-calo}{\skewchar\font=45 }
\DeclareFontShape{U}{BOONDOX-calo}{m}{n}{
    <-> s*[1.05] BOONDOX-r-calo}{}
\DeclareFontShape{U}{BOONDOX-calo}{b}{n}{
    <-> s*[1.05] BOONDOX-b-calo}{}
\DeclareMathAlphabet{\mathcalboondox}{U}{BOONDOX-calo}{m}{n}
\SetMathAlphabet{\mathcalboondox}{bold}{U}{BOONDOX-calo}{b}{n}
\DeclareMathAlphabet{\mathbcalboondox}{U}{BOONDOX-calo}{b}{n}

\titlespacing\section{0pt}{12pt plus 4pt minus 2pt}{3pt plus 2pt minus 2pt}
\titlespacing\subsection{0pt}{12pt plus 4pt minus 2pt}{3pt plus 2pt minus 2pt}
\titlespacing\subsubsection{0pt}{12pt plus 4pt minus 2pt}{3pt plus 2pt minus 2pt}

\titleformat{\section}[block]{\Large\bfseries}{\scshape\roman{section}}{1em}{}
\titleformat{\subsection}[block]{\large\bfseries}
                         {\scshape\roman{section}.\roman{subsection}}{0.7em}{}

\newcommand	\email[1]	{\href{mailto:#1}{\texttt{#1}}}

\newcommand	\vertexSet	      {\mathcalboondox V}
\newcommand	\edgeSet          {\mathcalboondox E}
\newcommand \secondVertexSet  {\mathcalboondox U}

\newcommand		\Matrix[1]	  {\left(#1_{ij}\right)}
\newcommand     \mat          \mathbf
\newenvironment{blockmatrix2} {\left(\begin{array}{c|c}}{\end{array}\right)}
\newenvironment{blockmatrix4} {\left(\begin{array}{c|c|c|c}}{\end{array}\right)}

\let\leq\relax
\newcommand\leq\leqslant
\let\geq\relax
\newcommand\geq\geqslant

\newcommand     \Reals  {{\mathds R}}
\newcommand     \Nats   {{\mathds N}}

\newcommand* \defeq		{\mathrel{\vcenter{\baselineskip0.5ex\lineskiplimit0pt\hbox{\scriptsize.}\hbox{\scriptsize.}}}=}

\DeclareMathOperator\spec{spec}

\DeclareMathOperator\Span{span}

\DeclareMathOperator\Main{Main}

\DeclareMathOperator\rank{rank}

\def \eigsp {\mathop{\scalebox{1.2}{\text{\fontfamily{pzc}\selectfont E}}}{\hspace{-1pt}}}

\DeclareMathOperator\CDC{CDC}

\newcommand \tfi {\stackrel{\text{\scriptsize TF}}{\simeq}}

\newcommand \transpose  {{}^{\scalebox{0.6}{\sffamily T}}}

\def\arraystretch{1.5}

\usepackage{mathtools}

\let \plaincite \cite
\let \cite \relax
\newcommand \cite[1] {\textsuperscript{\plaincite{#1}}}

\newdateformat{godsavethequeen}{\ordinal{DAY} \monthname[\THEMONTH], \THEYEAR}

\makeatletter
\let\save@mathaccent\mathaccent
\newcommand*\if@single[3]{%
    \setbox0\hbox{${\mathaccent"0362{#1}}^H$}%
    \setbox2\hbox{${\mathaccent"0362{\kern0pt#1}}^H$}%
    \ifdim\ht0=\ht2 #3\else #2\fi
}
\newcommand*\rel@kern[1]{\kern#1\dimexpr\macc@kerna}
\newcommand*\widebar [1]{\@ifnextchar^{{\wide@bar{#1}{0}}}{\wide@bar{#1}{1}}}
\newcommand*\wide@bar[2]{\if@single{#1}{\wide@bar@{#1}{#2}{1}}{\wide@bar@{#1}{#2}{2}}}
\newcommand*\wide@bar@[3]{%
    \begingroup
    \def\mathaccent##1##2{%
        \let\mathaccent\save@mathaccent
        \if#32 \let\macc@nucleus\first@char \fi
        \setbox\z@\hbox{$\macc@style{\macc@nucleus}_{}$}%
        \setbox\tw@\hbox{$\macc@style{\macc@nucleus}{}_{}$}%
        \dimen@\wd\tw@
        \advance\dimen@-\wd\z@
        \divide\dimen@ 3
        \@tempdima\wd\tw@
        \advance\@tempdima-\scriptspace
        \divide\@tempdima 10
        \advance\dimen@-\@tempdima
        \ifdim\dimen@>\z@ \dimen@0pt\fi
        \rel@kern{0.6}\kern-\dimen@
        \if#31
        \overline{\rel@kern{-0.6}\kern\dimen@\macc@nucleus\rel@kern{0.4}\kern\dimen@}%
        \advance\dimen@0.4\dimexpr\macc@kerna
        \let\final@kern#2%
        \ifdim\dimen@<\z@ \let\final@kern1\fi
        \if\final@kern1 \kern-\dimen@\fi
        \else
        \overline{\rel@kern{-0.6}\kern\dimen@#1}%
        \fi
    }%
    \macc@depth\@ne
    \let\math@bgroup\@empty \let\math@egroup\macc@set@skewchar
    \mathsurround\z@ \frozen@everymath{\mathgroup\macc@group\relax}%
    \macc@set@skewchar\relax
    \let\mathaccentV\macc@nested@a
    \if#31
    \macc@nested@a\relax111{#1}%
    \else
    \def\gobble@till@marker##1\endmarker{}%
    \futurelet\first@char\gobble@till@marker#1\endmarker
    \ifcat\noexpand\first@char A\else
    \def\first@char{}%
    \fi
    \macc@nested@a\relax111{\first@char}%
    \fi
    \endgroup
}
\makeatother

\newtheoremstyle{plain}
    {4pt}      
    {4pt}      
    {\itshape}  
    {0pt}       
    {\bfseries} 
    {.}         
    {5pt plus 1pt minus 1pt} 
    {}          
\theoremstyle{plain}
\newtheorem{lemma}{Lemma}[section]
\newtheorem{theorem}[lemma]{Theorem}
\newtheorem{prop}[lemma]{Proposition}
\newtheorem{corollary}[lemma]{Corollary}
\newtheorem{question}[lemma]{Question}

\newtheoremstyle{definition}
    {4pt}          
    {4pt}          
    {\normalfont}  
    {0pt}          
    {\bfseries}    
    {.}            
    {5pt plus 1pt minus 1pt} 
    {}             
\theoremstyle{definition}
\newtheorem{definition}[lemma]{Definition}

\newtheoremstyle{remark}
    {4pt}          
    {4pt}          
    {\normalfont}  
    {0pt}          
    {\itshape}     
    {.}            
    {5pt plus 1pt minus 1pt} 
    {}             
\theoremstyle{remark}

\newtheorem{example}[lemma]{Example}
\newtheorem{counterexample}[lemma]{Counterexample}
\AtEndEnvironment{remark}{\null\hfill\rule{1ex}{1ex}}
\AtEndEnvironment{example}{\null\hfill\rule{1ex}{1ex}}
\AtEndEnvironment{counterexample}{\null\hfill\rule{1ex}{1ex}}

\title{
	{\bfseries\scshape\Large
        On the Walks and Bipartite Double Coverings of Graphs with the same Main Eigenspace}
}

\author{
	\href{http://drmenguin.com}
	{\scshape Luke Collins}			\\[-8pt]
	{\normalsize\email{luke.collins@um.edu.mt}}
	\and
	\href{http://staff.um.edu.mt/isci1/}
	{\scshape Irene Sciriha}		\\[-8pt]
	{\normalsize\email{irene.sciriha-aquilina@um.edu.mt}}
}

\date {
	\href{https://www.um.edu.mt/science/maths}
	{Department of Mathematics},	\\
	\href{http://www.um.edu.mt/}
	{University of Malta},			\\
	Msida, Malta					\\[8pt]
	\godsavethequeen\today{} 
}

\begin{document}
	\maketitle

	\begin{abstract}
		\noindent
        The main eigenvalues of a graph $G$ are those eigenvalues of the $(0,1)$-adjacency matrix $\mat A$ having a corresponding eigenvector not orthogonal to $\mat j = (1,\dots,1)$. The CDC of a graph $G$ is the direct product $G\times K_2$.
        The main eigenspace of $\mat A$ is generated by the principal main eigenvectors and is the same as the image of the walk matrix. A hierarchy of properties of pairs of graphs is established in view of their CDC's, walk matrices, main eigenvalues, eigenvectors and eigenspaces. 
        We determine by algorithm that there are 32 pairs of non-isomorphic graphs on at most 8 vertices which have the same CDC.

		\vspace{1em}\noindent
		\textbf{Keywords:} Eigenvalues, walks, walk matrix, main eigenspace, canonical (bipartite) double covering, TF-isomorphism.
	\end{abstract}

	\section{Introduction}
	A graph of order $n$ is a pair of sets $G = (\vertexSet, \edgeSet)$ where $\vertexSet = \{1, \dots, n\}$ is called the set of vertices, and $\edgeSet \subseteq \{\{u,v\}: u, v \in \vertexSet \ \text{and}\ u \neq v\}$ is called the set of edges. (We consider graphs which are simple; that is, graphs which are undirected, without multiple edges or loops.) A $k$-walk in a graph $G$ is a $k$-tuple $(u_0, u_1, \dots, u_k) \in \vertexSet^{k+1}$ such that $\{u_{i-1}, u_{i}\} \in \edgeSet$ for all $1 \leq i \leq k$. 

    The adjacency matrix of a graph $G$, denoted by $\mat A(G)$, or simply $\mat A$ where the context is clear, is the symmetric $n\times n$ matrix $\Matrix a$, where $a_{ij} = 1$ if $\{i,j\}\in \edgeSet$, and $a_{ij} = 0$ otherwise. We use terminology for a graph $G$ and its adjacency matrix $\mat A$ interchangeably, since the graph $G$ is determined, up to relabelling of the vertices, by $\mat A$. For example, the eigenvalues and eigenvectors of a graph $G$ are respectively those of the matrix $\mat A$. The spectrum $\spec(G)$ of a graph $G$ is the multiset consisting of the $s$ distinct eigenvalues  $\mu_1, \dots, \mu_s$, each occurring $m(\mu_i)$ times, $1\leq i \leq s$; where the multiplicity $m(\mu_i)$ is the number of times that $\mu_i$ is repeated as a root of the characteristic polynomial $\det(\lambda\mat I - \mat A)$. Since $\mat A$ is real-symmetric, we also have that $m(\mu_i)$ is the dimension of the eigenspace $\eigsp_G(\mu_i)$ associated with $\mu_i$, where  $\eigsp_G(\mu_i) = \{\mat x\in\Reals^n : \mat A\mat x = \mu_i \mat x\}$, and $n=|\vertexSet|$. Spectral decomposition of $\mat A$ yields
    \begin{equation}
        \label{eqn:specdec}
        \mat A = \sum_{i=1}^s\mu_i\mat P_i,
    \end{equation}
    where $\mat P_i\colon\Reals\to\eigsp_G(\mu_i)$ is the orthogonal projection onto the eigenspace for $\mu_i$, $1 \leq i \leq s$.

    The entry $a_{ij}^{(k)}$ of the matrix $\mat A^k$ is the number of walks of length $k$ starting from vertex $i$ to vertex $j$. If $\mat j$ denotes the all-ones $n\times1$ column $(1,\dots,1)$, then the $i$th entry of $\mat A^k\mat j$ is the total number of walks of length $k$ starting from vertex $i$. The $n\times k$ matrix whose $k$ columns are $\mat A^{i-1}\mat j$ for $i=1,2,\dots,k$ is called the $k$-walk matrix of $G$, denoted by $\mat W_G(k)$:
    \[\mat W_G(k) = \begin{pmatrix}
    |&|&|&&|\\
    \mat j & \mat A\mat j & \mat A^2\mat j & \cdots & \mat A^{k-1}\mat j\\
    |&|&|&&|
    \end{pmatrix}. \]
    The eigenvalues $\mu_1, \mu_2, \dots, \mu_p$ of $G$ ($1 \leq p \leq n$) having an associated eigenvector \emph{not} orthogonal to $\mat j$ (i.e. $\langle\mat x,\mat j\rangle \neq 0$) are said to be \emph{main}. The remaining distinct eigenvalues $\mu_{p+1}, \mu_{p+2}, \dots, \mu_s$ ($s\leq n$) are \emph{non-main}. Walks and main eigenvalues are closely related---it turns out that the number of walks of length $k$ in $G$ is given by
    \[N_k = \sum_{i=1}^p \big\| \mat P_i\,\mat j \big\|^2\,{\mu_i}^k = c_1{\mu_1}^k+c_2{\mu_2}^2 + \cdots + c_p{\mu_p}^k, \]
    where $\mu_i$, $i=1,\dots, p$ are the main eigenvalues of $G$, $\mat P_i$ is as in \cref{eqn:specdec}, and $c_i \defeq \| \mat P_i\,\mat j \|^2$ are constants independent of the number $k$ (\plaincite[p.\ 44]{CvetkovicSpectra}). The eigenvector $\mat P_i\,\mat j$ of $\mu_i$ is called the \emph{principal main eigenvector} corresponding to $\mu_i$.

    A pair of graphs $G$ and $H$ are \emph{comain} if they have the same set of main eigenvalues (ignoring multiplicity). We denote the \emph{main eigenspace}, that is, the space generated by all principal main eigenvectors, by $\Main(G)$. Thus if $\mu_i$, $i=1,\dots,p$ are the main eigenvalues of $G$, then
    \begin{equation}
        \label{def:mainG}
        \Main(G) = \Span\{\mat P_1\,\mat j,\dots,\mat P_p\,\mat j\}.
    \end{equation}
    
    The disjoint union of the graphs $G_i = (\vertexSet_i, \edgeSet_i)$, $1 \leq i \leq k$, where each $G_i$  has order $n_i$, 
    denoted by $G_1 \cupdot \cdots \cupdot G_k$ or $\bigcupdot_{i=1}^k G_i$, is the graph $(\vertexSet, \edgeSet)$ of order $n_1+\cdots+n_k$ with vertex set $\vertexSet = \bigcup_{i=1}^k\vertexSet_i\times\{i\}$ and edge set \[\edgeSet = \{\{(u,i),(v,i)\} : \{u,v\} \in \edgeSet_i \}.\]

    \subsection{Overview of the Paper}
    In this paper, we provide a full characterisation of graphs in view of the following: their main eigenvalues, their main eigenspace, their walk matrix, and their canonical double covering, as illustrated in \cref{fig:charaterisation}.

    In \cref{sec:cdcs}, we define canonical double coverings (CDCs) and prove some basic results about them. In \cref{sec:walksAndSpaces}, we show how the walk matrix is related to graphs with the same CDC and with the same main eigenspace. In \cref{sec:tfisom}, we define TF-isomorphisms and prove that graphs with the same CDC are equivalent to TF-isomorphic graphs. In \cref{sec:hierarchy}, we present the hierarchy which relates the various common properties which pairs of graphs can have, such as having the same main eigenspace, having the same walk matrix, and having the same CDC. We also give counterexamples to various natural questions which arise in our discussion in cases where the converse of a result is false.
    
    
    \section{Canonical Double Coverings}
    \label{sec:cdcs}
    The \emph{canonical double covering} (also referred to as bipartite double covering in the literature) of a graph $G = (\vertexSet,\edgeSet)$ of order $n$, denoted by $\CDC(G)$, is a graph $G' = (\vertexSet',\edgeSet')$ of order $2n$ where $\vertexSet' = \vertexSet \times\{0,1\}$, and
    \[\edgeSet' = \big\{\{(u,0),(v,1)\}, \{(u,1),(v,0)\} : \{u,v\} \in \edgeSet \big\}.\]
    In other words, $\CDC(G)$ is obtained by producing two copies of the vertex set, and replacing edges $\{u,v\}$ in the original graph by edges from the first copy to the second copy, and vice versa (see \cref{fig:cdcEg} for examples). Clearly, $\CDC(G)$ is always bipartite, with partite sets $\vertexSet\times\{0\}$ and $\vertexSet\times\{1\}$.

    \begin{figure}
        \centering
        \tikzstyle{every node}=[circle, draw=black, fill=white, inner sep=0pt, minimum width=12pt, line width = 0.3mm]
        \begin{minipage}[b]{0.25\textwidth}
            \centering
            \begin{tikzpicture}[thick, scale=0.8]
                \draw  (0,0)    node (1) {\footnotesize \textsf 1};
                \draw  (1.5,0)  node (2) {\footnotesize \textsf 2};
                \draw  (60:1.5) node (3) {\footnotesize \textsf 3};

                \draw[line width = 0.3mm]{
                    (1) -- (2) -- (3) -- (1)
                };
            \end{tikzpicture}
            \vfil~
            \caption*{$C_3$}
        \end{minipage}
        \hfil
        \begin{minipage}[b]{0.74\textwidth}
            \centering
            \begin{tikzpicture}[thick, scale=0.8, square/.style={regular polygon,regular polygon sides=4, inner sep=.1em}]

                \newcommand \sidelength {1.5};
                \newcommand \shiftx {3.5};
                \newcommand \shifty {0};

                \draw  (0,0)            node (1) {\footnotesize \textsf 1};
                \draw  (\sidelength,0)  node (2) {\footnotesize \textsf 2};
                \draw  (60:\sidelength) node (3) {\footnotesize \textsf 3};

                \coordinate (shift) at (\shiftx, \shifty);
                \draw  ($(shift)+(0,0)$)            node[square] (1') {\footnotesize \textsf 1};
                \draw  ($(shift)+(\sidelength,0)$)  node[square] (2') {\footnotesize \textsf 2};
                \draw  ($(shift)+(60:\sidelength)$) node[square] (3') {\footnotesize \textsf 3};

                \draw[line width = 0.3mm]{
                    (1)  to[bend right=45] (2')
                    (1)  to[bend left=90]  (3')
                    (2)  to                (1')
                    (2)  to[bend left]     (3')
                    (3)  to                (1')
                    (3)  to[bend left=10]  (2')
                };
            \end{tikzpicture}
            \quad\raisebox{1.2cm}{\Large$\equiv$}\quad
            \begin{tikzpicture}[thick, scale=0.8, square/.style={regular polygon,regular polygon sides=4, inner sep=.1em}]
                \draw  (180:1.5)  node (1)          {\footnotesize \textsf 1};
                \draw  (120:1.5)  node[square] (2') {\footnotesize \textsf 2};
                \draw  (60:1.5)   node (3)          {\footnotesize \textsf 3};
                \draw  (0:1.5)    node[square] (1') {\footnotesize \textsf 1};
                \draw  (-60:1.5)  node (2)          {\footnotesize \textsf 2};
                \draw  (-120:1.5) node[square] (3') {\footnotesize \textsf 3};

                \draw[line width = 0.3mm]{
                    (1) -- (2') -- (3) -- (1') -- (2) -- (3') -- (1)
                };
            \end{tikzpicture}
            \caption*{$\CDC(C_3) \simeq C_6$}
        \end{minipage}
         \begin{minipage}[b]{0.25\textwidth}
            \centering
            \begin{tikzpicture}[thick, scale=0.8]
                \draw  (60:1.5)             node (1) {\footnotesize \textsf 1};
                \draw  ($(1.5,0)+(60:1.5)$) node (2) {\footnotesize \textsf 2};
                \draw  (0,0)                node (3) {\footnotesize \textsf 3};
                \draw  (1.5,0)              node (4) {\footnotesize \textsf 4};
                \draw  (2*1.5,0)            node (5) {\footnotesize \textsf 5};

                \draw[line width = 0.3mm]{
                    (1) edge (3) edge (4) edge (5)
                    (2) edge (3) edge (4) edge (5)
                };
            \end{tikzpicture}
            \vfil~
            \caption*{$K_{2,3}$}
        \end{minipage}
        \hfil
        \begin{minipage}[b]{0.74\textwidth}
            \centering
            \begin{tikzpicture}[thick, scale=0.8, square/.style={regular polygon,regular polygon sides=4, inner sep=.1em}]

                \newcommand \shiftx {1.5};
                \newcommand \shifty {-3};

                \draw  (60:1.5)             node (1) {\footnotesize \textsf 1};
                \draw  ($(1.5,0)+(60:1.5)$) node (2) {\footnotesize \textsf 2};
                \draw  (0,0)                node (3) {\footnotesize \textsf 3};
                \draw  (1.5,0)              node (4) {\footnotesize \textsf 4};
                \draw  (2*1.5,0)            node (5) {\footnotesize \textsf 5};

                \coordinate (shift) at (\shiftx, \shifty);
                \draw  ($(shift) + (60:1.5)$)         node[square] (1') {\footnotesize \textsf 1};
                \draw  ($(shift) + (1.5,0)+(60:1.5)$) node[square] (2') {\footnotesize \textsf 2};
                \draw  ($(shift) +(0,0)$)             node[square] (3') {\footnotesize \textsf 3};
                \draw  ($(shift) +(1.5,0)$)           node[square] (4') {\footnotesize \textsf 4};
                \draw  ($(shift) +(2*1.5,0)$)         node[square] (5') {\footnotesize \textsf 5};

                \draw[line width = 0.3mm]{
                    (1') edge (3) edge (4) edge (5)
                    (2') edge (3) edge (4) edge (5)
                    (1) edge[bend right=20] (3') edge[bend right=25] (4') .. controls (5,3.5) and (4.5,-2) .. (5')
                    (2) edge (3') edge (4') edge[bend left] (5')
                };
            \end{tikzpicture}
            \quad\raisebox{1.5cm}{\Large$\equiv$}\quad
            \begin{tikzpicture}[thick, scale=0.8, square/.style={regular polygon,regular polygon sides=4, inner sep=.1em}]

                \newcommand \shiftx {0};
                \newcommand \shifty {-2.5};

                \draw  (60:1.5)             node (1) {\footnotesize \textsf 1};
                \draw  ($(1.5,0)+(60:1.5)$) node (2) {\footnotesize \textsf 2};
                \draw  (0,0)                node[square] (3) {\footnotesize \textsf 3};
                \draw  (1.5,0)              node[square] (4) {\footnotesize \textsf 4};
                \draw  (2*1.5,0)            node[square] (5) {\footnotesize \textsf 5};

                \coordinate (shift) at (\shiftx, \shifty);
                \draw  ($(shift) + (60:1.5)$)         node[square] (1') {\footnotesize \textsf 1};
                \draw  ($(shift) + (1.5,0)+(60:1.5)$) node[square] (2') {\footnotesize \textsf 2};
                \draw  ($(shift) +(0,0)$)             node         (3') {\footnotesize \textsf 3};
                \draw  ($(shift) +(1.5,0)$)           node         (4') {\footnotesize \textsf 4};
                \draw  ($(shift) +(2*1.5,0)$)         node         (5') {\footnotesize \textsf 5};

                \draw[line width = 0.3mm]{
                    (1) edge (3) edge (4) edge (5)
                    (2) edge (3) edge (4) edge (5)
                    (1') edge (3') edge (4') edge (5')
                    (2') edge (3') edge (4') edge (5')
                };
            \end{tikzpicture}
            \caption*{$\CDC(K_{2,3}) \simeq K_{2,3} \cupdot K_{2,3}$}
        \end{minipage}
        \caption{Canonical double coverings of $C_3$ and $K_{2,3}$, where vertices $(v,0)$ are represented by circle nodes, and vertices $(v,1)$ by square nodes.}
        \label{fig:cdcEg}
    \end{figure}

    If the vertices in $\vertexSet\times\{0\}$ are given the first $n$ labels, it is not hard to see that the adjacency matrix of $\CDC(G)$ is given by
    \[\mat A(\CDC(G)) = \begin{blockmatrix2}
        \mat O & \mat A(G)\\
        \hline
        \mat A(G) & \mat O
    \end{blockmatrix2}.\]
    This is actually equivalent to the direct product with $K_2$, i.e., $\CDC(G) = G \times K_2$. It can also be obtained as the NEPS of $G$ and $K_2$ with basis $\{(1,1)\}$.\cite{CvetkovicSpectra} Consequently, the eigenvalues of $\CDC(G)$ are those of $G$ and their negatives; i.e.
    \[\spec(\CDC(G)) = \pm\spec(G).\]

    The following result distinguishes between bipartite and non-bipartite connected graphs.
    \begin{lemma}
        Let $G$ be a connected graph. Then $G$ is bipartite if and only if $\CDC(G)$ is disconnected. Moreover, if $G$ is bipartite, then $\CDC(G) \simeq G \cupdot G$.
    \end{lemma}

    \begin{proof}
        Let $G$ be bipartite, and let $\secondVertexSet_1$, $\secondVertexSet_2$ be the partite sets of $G$. Consider $\CDC(G)$, and let $\vertexSet_i = \{(v,0) : v \in \secondVertexSet_i\}$ and $\vertexSet_i' = \{(v,1) : v\in \secondVertexSet_i\}$ for $i=1,2$ be the corresponding partite sets and their copies in $\CDC(G)$. Since edges in $G$ are only from $\secondVertexSet_1$ to $\secondVertexSet_2$, then edges in $\CDC(G)$ are only either from $\vertexSet_1$ to $\vertexSet_2'$ or $\vertexSet_2$ to $\vertexSet_1'$. Therefore $\CDC(G)$ is disconnected with components being precisely the induced subgraphs on $\vertexSet_1 \cup \vertexSet_2'$ and $\vertexSet_2 \cup \vertexSet_1'$, both of which are isomorphic to $G$.

        For the converse, suppose $\CDC(G)$ is connected. Let $v_1 \equiv (v_1,0)$ and $v_1'\equiv (v_1,1)$ denote the two copies in $\CDC(G)$ of a vertex $v_1$ in $G$. Since $\CDC(G)$ is connected, there is a path $v_1\to v_2'\to v_3\to\cdots\to v_{k-1}' \to v_k\to v_1'$ joining $v_1$ to $v_1'$, where the vertices alternate from one copy of the vertex set to another. But this corresponds to the odd cycle $v_1 \to v_2 \to v_3 \to \cdots \to v_k \to v_1$ in $G$. Hence $G$ is not bipartite.
    \end{proof}

    Next we prove that the $\CDC$ operation is additive with respect to disjoint union.

    \begin{lemma}
        \label{lemma:distributivityCDC}
        Let $G$ be disconnected with components $G_1, \dots, G_k$, so that $G = G_1\cupdot \cdots \cupdot G_k$ and $G_i$ is connected for $1 \leq i \leq k$. Then
        \[\CDC(G) = \CDC\left(\bigcupdot_{i=1}^k G_i\right) \simeq \bigcupdot_{i=1}^k\CDC(G_i).\]
    \end{lemma}

    \begin{proof}
        We give the proof for $k=2$, the general case follows by induction on $k$. If $G$ has components $G_1$ and $G_2$, then labelling the vertices of $G_1$ first gives us that the adjacency matrix of $G$ has block form
        \begingroup
        \renewcommand{\arraystretch}{1.4}
        \[\mat A(G) = \begin{blockmatrix2}
            \mat A(G_1) & \mat O\\
            \hline
            \mat O      & \mat A(G_2)
        \end{blockmatrix2},\]
        and so
        \begin{equation}
        \label{eq:cdcForm}
          \mat A\big(\CDC(G)\big) = \begin{blockmatrix4}
                \multicolumn{2}{c|}{\multirow{2}{*}{$\mat O$}}  & \mat A(G_1)           & \mat O                    \\ \cline{3-4}
                \multicolumn{2}{c|}{}                       & \mat O                & \mat A(G_2)               \\ \hline
                \mat A(G_1)             & \mat O                    & \multicolumn{2}{c}{\multirow{2}{*}{$\mat O$}} \\ \cline{1-2}
                \mat O                  & \mat A(G_2)               & \multicolumn{2}{c}{}
           \end{blockmatrix4}.
       \end{equation}
        On the other hand, for $i = 1,2$, we have
        \[\mat A\big(\CDC(G_i)\big) = \begin{blockmatrix2}
            \mat O      & \mat A(G_i)\\
            \hline
            \mat A(G_i) & \mat O
        \end{blockmatrix2}, \]
        so that
        \begin{equation}
            \label{eq:unionForm}
            \mat A\big(\CDC(G_1) \cupdot \CDC(G_2)\big) =
            \begin{blockmatrix4}
                \mat O                  & \mat A(G_1)               & \multicolumn{2}{c}{\multirow{2}{*}{$\mat O$}}     \\ \cline{1-2}
                \mat A(G_1)             & \mat O                    & \multicolumn{2}{c}{}                          \\ \hline
                \multicolumn{2}{c|}{\multirow{2}{*}{$\mat O$}}  & \mat O                    & \mat A(G_2)               \\ \cline{3-4}
                \multicolumn{2}{c|}{}                       & \mat A(G_2)               & \mat O
            \end{blockmatrix4}.
        \end{equation}
        \endgroup
        When considering \cref{eq:cdcForm,eq:unionForm} It is not hard to see that the permutation matrix \[\mat P = \begin{blockmatrix4}
            \mat I_1 & \mat O & \mat O & \mat O \\
            \hline
            \mat O & \mat O & \mat I_2 & \mat O \\
            \hline
            \mat O & \mat I_1 & \mat O & \mat O \\
            \hline
            \mat O & \mat O & \mat O & \mat I_2
        \end{blockmatrix4},\]
        where $\mat I_i$ denotes the $|\vertexSet(G_i)|\times|\vertexSet(G_i)|$ identity matrix, gives the required relabelling:
        \[\mat P\transpose\mat A\big(\CDC(G)\big)  \mat P =  \mat A\big(\CDC(G_1) \cupdot \CDC(G_2)\big), \]
        so that $\CDC(G) \simeq \CDC(G_1) \cupdot \CDC(G_2)$, as required.
    \end{proof}

    \section{The Walk Matrix and Main Eigenspace}
    \label{sec:walksAndSpaces}
    
    Recall that the $k$-walk matrix $\mat W_G(k)$ is the matrix with columns $\mat A^i\mat j$ for $i = 0,\dots, k-1$, where $\mat A$ is the adjacency matrix of $G$, and $\mat j = (1,\dots,1)$.
    
    \begin{definition}[Walk Matrix]
        The \emph{walk matrix} of a graph $G$ having $p$ distinct main eigenvalues, denoted by $\mat W_G$, is the $p$-walk matrix of $G$. In other words, $\mat W_G = \mat W_G(p)$.
    \end{definition}
    
    In \plaincite{powerSuleiman}, the authors show that the first $p$ columns suffice to generate $\mat W_G(k)$ for $k \geq p$. Suppose $\mu_1, \cdots, \mu_p$ are the main eigenvalues of $G$. If one forms the \emph{main characteristic polynomial} $m_G(x) = \prod_{i=1}^p (x-\mu_i) = x^p - c_{0}x^{p-1} - \cdots - c_{p-2}x - c_{p-1}$, then
    \[(\mat A^p - c_0\mat I - c_1\mat A - \cdots - c_{p-1}\mat A^{p-1})\mat j = \prod_{i=1}^p (\mat A-\mu_i\mat I)\sum_{j=1}^p \mat P_j\mat j = \mat 0. \]
    This gives a recurrence relation for the $k$th column of $\mat W_G(k)$ in terms of the previous $p$ columns.    
    Consequently, any two comain graphs with the same walk matrix have the same $k$-walk matrix for any $k\geqslant p$. 
    
    \begin{counterexample}
        \label{cex:sameW}
        Unfortunately one may not extend a walk matrix $\mat W_G(p)=\mat W_H(p)$ common to two non-comain graphs $G$ and $H$ to the same $k$-walk matrix for arbitrary $k\geq p$. The two pairs $(G_{5\,622},G_{12\,058})$ and $(G_{5\,626},G_{12\,093})$ in \cref{fig:sameWcex} are the only two smallest  counterexample pairs (with respect to the number of vertices), obtained using Mathematica.
        They  are the only counterexamples on at most $8$ vertices having the same walk matrix, but not the same $k$-walk matrix for $k\geq p$. 
        
        The numbering of the graphs is in accordance with the list of non-isomorphic graphs on 8 vertices provided on Brendan McKay's graph data website.\cite{McKay}       
        \begin{figure}
            \begingroup
            \renewcommand{\arraystretch}{0.9}
            \centering
            \begin{minipage}{0.45\textwidth}
                \centering
                \includegraphics[height=4cm]{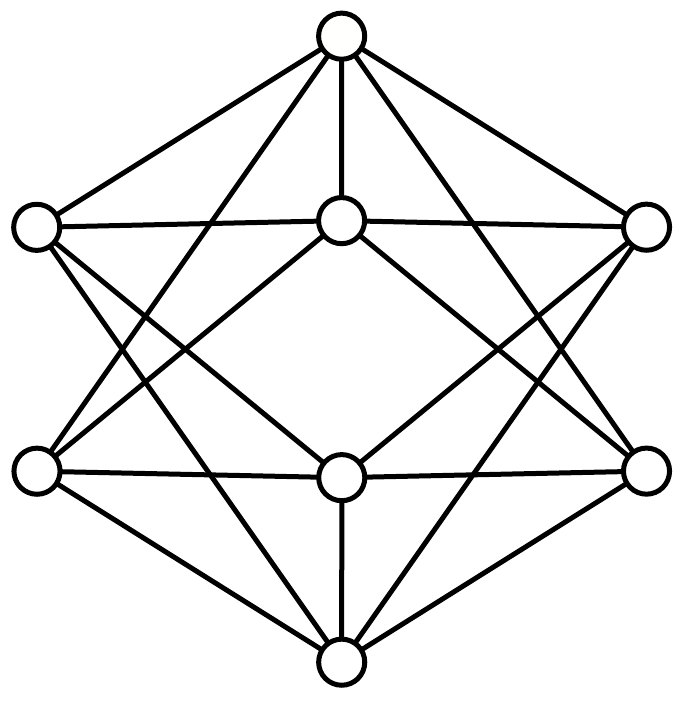}\\
                $G_{5\,622}$\\[4pt]                
                {\footnotesize Main Eigenvalues: $\frac{1-\sqrt{65}}2,\frac{1+\sqrt{65}}{2}$\\[10pt]                
                    \begin{tabular}{cc}
                        Walk Matrix & $3$-walk Matrix\\[4pt]
                        $\begin{pmatrix}
                        1&4\\
                        1&4\\
                        1&4\\
                        1&4\\
                        1&5\\
                        1&5\\
                        1&5\\
                        1&5
                        \end{pmatrix}$&
                        $\begin{pmatrix}
                        1&4&20\\
                        1&4&20\\
                        1&4&20\\
                        1&4&20\\
                        1&5&21\\
                        1&5&21\\
                        1&5&21\\
                        1&5&21
                        \end{pmatrix}$
                    \end{tabular}    
                } 
            \end{minipage} \hfil 
            \begin{minipage}{0.45\textwidth}
                \centering
                \vspace{0.5cm}
                \includegraphics[height=3cm]{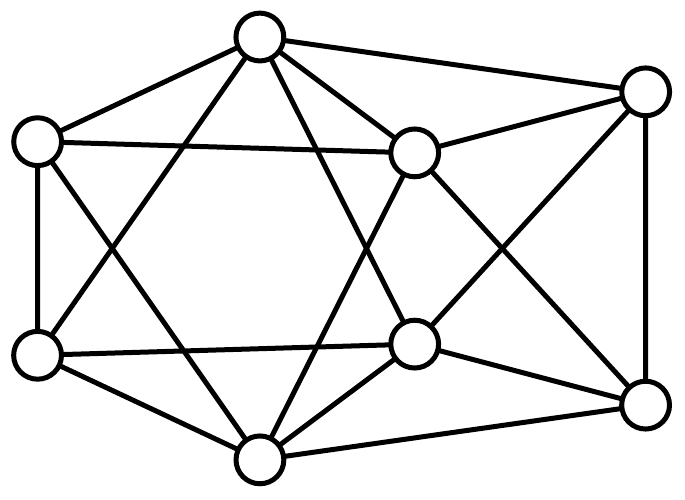}
                \vspace{0.5cm}\\
                $G_{12\,058}$\\[4pt]                
                {\footnotesize Main Eigenvalues: $\frac{3-\sqrt{37}}2,\frac{3+\sqrt{37}}{2}$\\[10pt] 
                    \begin{tabular}{cc}
                        Walk Matrix & $3$-walk Matrix\\[4pt]
                        $\begin{pmatrix}
                        1&4\\
                        1&4\\
                        1&4\\
                        1&4\\
                        1&5\\
                        1&5\\
                        1&5\\
                        1&5
                        \end{pmatrix}$&
                        $\begin{pmatrix}
                        1&4&19\\
                        1&4&19\\
                        1&4&19\\
                        1&4&19\\
                        1&5&22\\
                        1&5&22\\
                        1&5&22\\
                        1&5&22
                        \end{pmatrix}$
                    \end{tabular}               
                } 
            \end{minipage}\\[20pt]
            
            \begin{minipage}{0.45\textwidth}
                \centering
                \includegraphics[height=3cm]{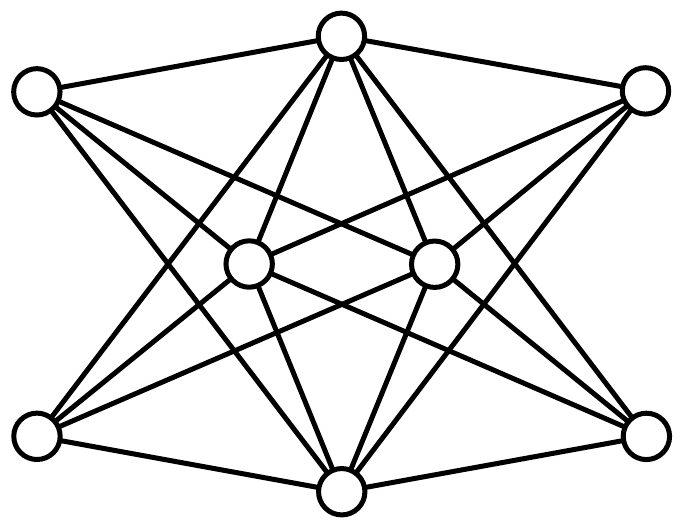}\\
                $G_{5\,626}$\\[4pt]                
                {\footnotesize Main Eigenvalues: $1+\sqrt{17}, 1-\sqrt{17}$\\[10pt]                
                    \begin{tabular}{cc}
                        Walk Matrix & $3$-walk Matrix\\[4pt]
                        $\begin{pmatrix}
                        1&4\\
                        1&4\\
                        1&4\\
                        1&4\\
                        1&6\\
                        1&6\\
                        1&6\\
                        1&6
                        \end{pmatrix}$&
                        $\begin{pmatrix}
                        1&4&24\\
                        1&4&24\\
                        1&4&24\\
                        1&4&24\\
                        1&6&28\\
                        1&6&28\\
                        1&6&28\\
                        1&6&28
                        \end{pmatrix}$
                    \end{tabular}    
                } 
            \end{minipage} \hfil 
            \begin{minipage}{0.45\textwidth}
                \centering
                \includegraphics[height=3cm]{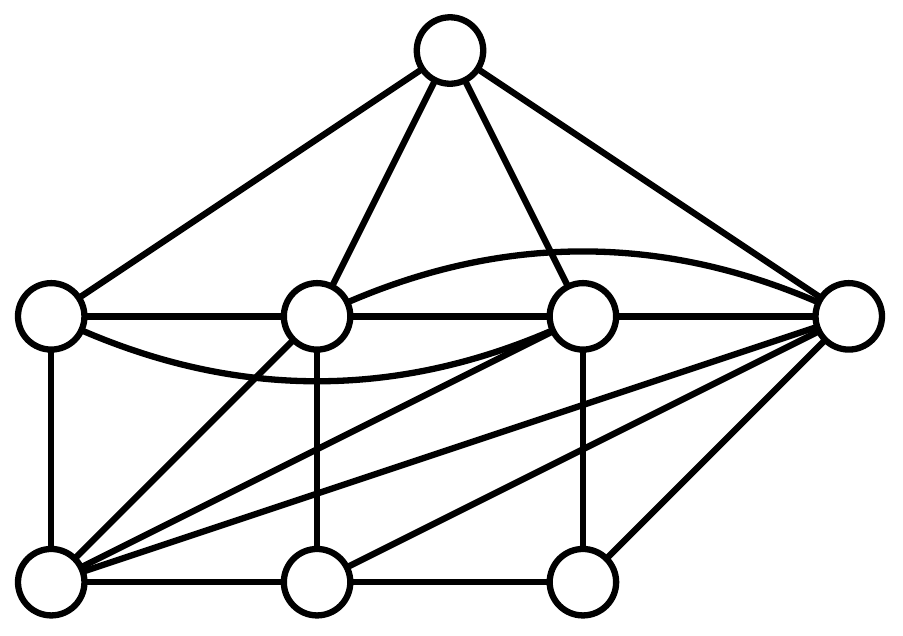}\\
                $G_{12\,093}$\\[4pt]                
                {\footnotesize Main Eigenvalues: $2+\sqrt{10}, 2-\sqrt{10}$\\[10pt] 
                    \begin{tabular}{cc}
                        Walk Matrix & $3$-walk Matrix\\[4pt]
                        $\begin{pmatrix}
                        1&4\\
                        1&4\\
                        1&4\\
                        1&4\\
                        1&6\\
                        1&6\\
                        1&6\\
                        1&6
                        \end{pmatrix}$&
                        $\begin{pmatrix}
                        1&4&22\\
                        1&4&22\\
                        1&4&22\\
                        1&4&22\\
                        1&6&30\\
                        1&6&30\\
                        1&6&30\\
                        1&6&30
                        \end{pmatrix}$
                    \end{tabular}               
                } 
            \end{minipage}
            \caption{The only two counterexamples on at most $8$ vertices, as described in \cref{cex:sameW}.}
            \label{fig:sameWcex}
            \endgroup
        \end{figure}
    \end{counterexample}
    
    \begin{theorem}
        \label{thm:TFisom=>SameW}
        Let $G$, $H$ be two graphs with $\CDC(G)\simeq \CDC(H)$, and let $k$ be a natural number. Then \[\mat W_G(k) = \mat W_H(k)\]
        for appropriate labelling of the vertices.
    \end{theorem}
    
    \begin{proof}
        For a graph $\Gamma$, let $\mat A_\Gamma = \mat A(\Gamma)$ and $\mat C_\Gamma = \mat A(\CDC(\Gamma))$. Since $\CDC(G)\simeq\CDC(H)$, we can relabel the vertices of the graph $H$ to get $H'$, so that $\mat C_G = \mat C_{H'}$. Now for any $0\leq\ell\leq k$, we have that
        \[{\mat C_G}^\ell\mat j = \begin{pmatrix}
            {\mat A_G}^\ell\mat j\\
            \hline
            {\mat A_G}^\ell\mat j
        \end{pmatrix} \qquad  \text{and} \qquad
        {\mat C_{H'}}^\ell\mat j = \begin{pmatrix}
        {\mat A_{H'}}^\ell\mat j\\
        \hline
        {\mat A_{H'}}^\ell\mat j
        \end{pmatrix},  \]
        but since $\mat C_G = \mat C_{H'}$, it follows that ${\mat A_G}^\ell\mat j = {\mat A_{H'}}^\ell\mat j$ for all $0\leq\ell\leq k$, so the columns of $\mat W_G(k)$ and $\mat W_H(k)$ are equal.
    \end{proof}

    \begin{counterexample}
        \label{cex:SameW=>TFisom}
        A counterexample establishing that the converse of \cref{thm:TFisom=>SameW} is false is given in \cref{fig:counterexampleConverseCDCW}. Indeed, those graphs have
        \begingroup
        \renewcommand{\arraystretch}{0.8}
        \[\mat W_G = \begin{pmatrix}
        1& 3& 9\\
        1& 3& 10\\
        1& 3& 10\\
        1& 3& 10\\
        1& 3& 10\\
        1& 3& 9\\
        1& 4& 12
        \end{pmatrix}= \mat W_H,\]
        \endgroup
        but $\CDC(G) \nsimeq \CDC(H)$.
    \end{counterexample}
    
    \begin{figure}
        \centering
        \tikzstyle{every node}=[circle, draw=black, fill=white, inner sep=0pt, minimum width=8pt, line width = 0.3mm]
        \begin{minipage}[b]{0.5\textwidth}
            \centering
            \begin{tikzpicture}[thick, scale=0.8]
            \draw  (0,4)   node (1) {};
            \draw  (2,4)   node (2) {};
            \draw  (4,4)   node (3) {};
            \draw  (0,2)   node (4) {};
            \draw  (2,2)   node (5) {};
            \draw  (4,2)   node (6) {};
            \draw  (2,0.5) node (7) {};
            
            \draw[line width = 0.3mm]{
                (1) -- (2) -- (3) -- (6) -- (2) -- (4) -- (1)
                (1) -- (5) -- (3)
                (4) -- (7) -- (6)
                (5) -- (7)
            };
            \end{tikzpicture}
            \caption*{Graph $G$}
        \end{minipage}
        \begin{minipage}[b]{0.49\textwidth}
            \centering
            \begin{tikzpicture}[thick, scale=0.8]
            \draw  (0,4)   node (1) {};
            \draw  (2,4)   node (2) {};
            \draw  (4,4)   node (3) {};
            \draw  (0,2)   node (4) {};
            \draw  (1.2,1) node (5) {};
            \draw  (4,2)   node (6) {};
            \draw  (2.8,1) node (7) {};
            
            \draw[line width = 0.3mm]{
                (1) -- (2) -- (3) -- (6) -- (2) -- (4) -- (1)
                (1) -- (5) -- (4)
                (3) -- (7) -- (6)
                (5) -- (7)
            };
            \end{tikzpicture}
            \vfil
            \caption*{Graph $H$}
        \end{minipage}
        \caption{Graphs $G$ and $H$ give a counterexample to the converse of \cref{thm:TFisom=>SameW}, since they have the same walk matrix but different CDC's.}
        \label{fig:counterexampleConverseCDCW}
    \end{figure}
    
    Recall that the main eigenspace $\Main(G)$ of a graph $G$ is the subspace generated by the vectors obtained by the decomposition of $\mat j$ into the eigenspaces of $\mat A(G)$. An important result about $\Main(G)$ which follows from Vandermonde matrix theory is given in \plaincite{ScCardHam12} and \plaincite{RowlinsonSurvey}, and states the following.
    \begin{theorem}[\plaincite{ScCardHam12}]
        \label{thm:colsOfWareBasis}
        Let $G$ be a graph with adjacency matrix $\mat A$. Then the set \[B=\{\mat j, \mat{Aj},\dots,\mat{A}^{p-1}\mat j\}\] is a set of linearly independent vectors in $\Reals^n$. Moreover, $B$ forms a basis for $\Main(G)$.
    \end{theorem}
    This fact yields the following two corollaries immediately.
    \begin{corollary}[\plaincite{hagos}]
        Let $G$ be a graph. Then \[\dim(\Main(G))=p,\qquad \text{and}\qquad \rank\mat W_G(k) = \min\{k,p\}.\]
    \end{corollary}
    
    \begin{corollary}[\plaincite{ScCardHam12}]
        \label{corr:sameW=>sameMainSp}
        Let $G$ and $H$ be two graphs with the same walk matrix. Then \[\Main(G) = \Main(H).\]
    \end{corollary}

    \section{Equivalence of Same CDC and TF-Isomorphism}
    \label{sec:tfisom}
    If two graphs $G$, $H$ have isomorphic canonical double coverings, that is, $\CDC(G) \simeq \CDC(H)$, and $G$ is connected, we do not necessarily have that $H$ is connected. Indeed, $\CDC(C_6) \simeq \CDC(K_3 \cupdot K_3)$ as we saw in \cref{fig:cdcEg}.
    
    However, we do have the following.

    \begin{lemma}
         \label{lemma:NoIsolatedVertices}
         Let $G$ and $H$ be two graphs with $\CDC(G) \simeq \CDC(H)$.  Then
         $G$ has no isolated vertices if and only if $H$ has no isolated vertices.
    \end{lemma}

    \begin{proof}
        Indeed, if $G$ has an isolated vertex, then $G = G' \cupdot K_1$, so \[\CDC(G) = \CDC(G' \cupdot K_1) = \CDC(G') \cupdot \CDC(K_1) = \CDC(G') \cupdot\widebar K_2\]
        by \cref{lemma:distributivityCDC}, and therefore $\CDC(H) = \CDC(G') \cupdot \widebar K_2$. Thus the matrix
        \[\mat A(\CDC(H)) = \begin{blockmatrix2}
        \mat O&\mat A(H)\\
        \hline
        \mat A(H)&\mat O
        \end{blockmatrix2}, \]
        has two whole columns of zeros, corresponding to the isolated vertices which make up $\widebar K_2$. But a column of zeros in the matrix above arises only when a whole column of zeros is present in one of the non-zero blocks $\mat A(H)$, and since both non-zero blocks are equal to $\mat A(H)$, then these two columns must be distributed equally among both $\mat A(H)$'s (otherwise they would be different). In other words, $\mat A(H)$ must have a column of zeros, and consequently $H$ has an isolated vertex. This argument is symmetric (simply interchange $G$ and $H$) so we also have the converse.
    \end{proof}

    \begin{theorem}
        \label{thm:QGR=H}
        Suppose $G$ and $H$ are two graphs with adjacency matrices ${\mat A}_G$ and ${\mat A}_H$. Then $\CDC(G) \simeq \CDC(H)$ if and only if there exist two permutation matrices $\mat Q$ and $\mat R$ such that
        \[\mat Q\, {\mat A}_G\, \mat R = {\mat A}_H. \]
    \end{theorem}

    \begin{proof}
        Suppose, without loss of generality, that the graphs $G$ and $H$ have no isolated vertices (if they do, then by \cref{lemma:NoIsolatedVertices}, we could simply pair them off until we are left with two graphs having no isolated vertices). If $\CDC(G) \simeq \CDC(H)$, then there exists a permutation matrix 
        $\mat P = \left(\begin{smallmatrix}
            \multicolumn{1}{c|}{$\mat P_{11}$} & \mat P_{12}\\[2pt]
            \hline
            \multicolumn{1}{c|}{}&\\
            \multicolumn{1}{c|}{$\mat P_{21}$} & \mat P_{22}
        \end{smallmatrix} \right)$
        such that
        \begin{gather*}
            \mat P\transpose\begin{blockmatrix2}
            \mat O & \mat A_G\\
            \hline
            \mat A_G & \mat O
            \end{blockmatrix2}\mat P = \begin{blockmatrix2}
            \mat O & \mat A_H\\
            \hline
            \mat A_H & \mat O
            \end{blockmatrix2}.  
        \end{gather*}
        Hence by comparing entries, we get that
        \begin{gather}
            \label{eqn:perms1}
            \mat P_{21}\transpose\mat A_G{\mat P_{12}} + \mat P_{11}\transpose\mat A_G{\mat P_{22}} = \mat A_H\\
            \label{eqn:perms2}
            \mat P_{21}\transpose\mat A_G{\mat P_{11}} = \mat P_{12}\transpose\mat A_G{\mat P_{22}} = \mat O,
        \end{gather}
        where \cref{eqn:perms2} follows since all the matrices have non-negative entries.

        Now observe that
        \[(\mat P_{11}+\mat P_{21})\transpose \mat A_G (\mat P_{22}+\mat P_{12}) = \mat A_{H} \]
        by \cref{eqn:perms1,eqn:perms2}. Now suppose $\mat Q = (\mat P_{11}+\mat P_{21})\transpose$ or $\mat R = \mat P_{22}+\mat P_{12}$ is not a permutation matrix. Being the sum of two submatrices of $\mat P$, this can only happen if a row (and column) are zero. But then $\mat A_H$ will have a row of zeros, corresponding to an isolated vertex in $H$, a contradiction.

        Conversely, if $\mat Q \mat A_G \mat R = \mat A_H$, then clearly $\mat P \defeq \begin{blockmatrix2}
            \mat O & \mat Q\phantom\transpose\\
            \hline
            \mat R\transpose & \mat O
        \end{blockmatrix2}$ defines a permutation matrix, and it is easy to verify that
        \[\mat P\transpose \begin{blockmatrix2}
            \mat O & \mat A_G\\
            \hline
            \mat A_G & \mat O
        \end{blockmatrix2} \mat P = \begin{blockmatrix2}
            \mat O & \mat A_H\\
            \hline
            \mat A_H & \mat O
        \end{blockmatrix2},\]
        as required.
    \end{proof}

    This weakened notion of graph isomorphism, where $\mat Q \mat A_G \mat R = \mat A_H$ and the permutation matrices $\mat Q$ and $\mat R$ are not necessarily inverses, was first studied by Lauri et al.\ in \plaincite{Lauri}. They give a different proof of \cref{thm:QGR=H} which uses a combinatorial argument. Such graphs are said to be \emph{two-fold isomorphic} or \emph{TF-isomorphic}, and we write \[G \tfi H.\]
    The pair of permutations $(\mat Q, \mat R)$ is called the \emph{TF-isomorphism}\index{TF-isomorphism}.
    
    In \plaincite{Lauri}, the authors introduced TF-isomorphisms. They discuss a pair of TF-isomorphic graphs on 7 vertices found by B.\ Zelinka (\cref{fig:zelinka}). This is the third out of the 32 pairs we found using Mathematica.
    
    \begin{figure}
        \centering
        \tikzstyle{every node}=[circle, draw=black, fill=white, inner sep=0pt, minimum width=8pt, line width = 0.3mm]
        \begin{minipage}[b]{0.5\textwidth}
            \centering
            \begin{tikzpicture}[thick, scale=1.2]
                \draw  (-1.8,0.5)   node (1) {};
                \draw  (-1.8,-0.5)  node (2) {};
                \draw  (-1,0)       node (3) {};
                \draw  (0,0)        node (4) {};
                \draw  (1,0)        node (5) {};
                \draw  (1.8,0.5)    node (6) {};
                \draw  (1.8,-0.5)   node (7) {};
                
                \draw[line width = 0.3mm]{
                    (1) -- (2) -- (3) -- (1)
                    (3) -- (4) -- (5) -- (6) -- (7) -- (5)
                };
            \end{tikzpicture}
            \vspace{0.5cm}
            \caption*{Graph $G$}
        \end{minipage}
        \begin{minipage}[b]{0.49\textwidth}
            \centering
            \begin{tikzpicture}[thick, scale=0.8]
                \draw  (-1.8,1)  node (1) {};
                \draw  (0,1.4)   node (2) {};
                \draw  (1.8,1)   node (3) {};
                \draw  (1.8,-1)  node (4) {};
                \draw  (0,-1.4)  node (5) {};
                \draw  (-1.8,-1) node (6) {};
                \draw  (0,0)     node (7) {};
                
                \draw[line width = 0.3mm]{
                    (1) -- (2) -- (3) -- (4) -- (5) -- (6) -- (1)
                    (2) -- (7)
                    (7) -- (5)
                };
            \end{tikzpicture}
            \vfil
            \caption*{Graph $H$}
        \end{minipage}
        \caption{The Zelinka example discussed in \plaincite{Lauri}.}
        \label{fig:zelinka}
    \end{figure}

%
%
%

    \section{Establishing the Hierarchy}
    \label{sec:hierarchy}
    In this final section, we compare the strength of relationships and similarities between graphs using the results presented above. This establishes a hierarchy depending on their main eigenvalues, main eigenspaces, main eigenvalues, walk matrices, and CDCs. 
    
    \begin{figure}[h]
        \centering
        \footnotesize
        \begin{tikzpicture}    
        \tikzset{every node/.style={draw=black, align=center, text width=2.1cm}} 
        
        \node                                 (smes) {Same Main Eigenspace};
        \node [below right = of smes]         (rwm)  {Related Walk Matrices};
        \node [below left  = of smes]         (smev) {Same Main Eigenvectors};
        \node [below       = of smev]         (swm)  {Same Walk Matrix};
        \node [below       = of swm]          (scdc) {Isomorphic CDCs};
        \node [left        = 1.2cm of scdc]   (cmn)  {Same Main Eigenvalues};
        \node [below       = of scdc]         (tfis) {Two-fold isomorphic};
        
        \tikzset{every node/.style={}, 
            every path/.style={>=implies, double equal sign distance}}
        
        \node [circle,draw, below left = 0cm and 2.125cm of smev]     (join) {$\land$};
        
        \draw[<->] (smes.330) -- node[above right]{\ref{thm:relatedWalkMatrices}} (rwm.160);
        \draw[<- ] (smes.210) -- node[below right]{\eqref{def:mainG}} (smev.25);       
        \draw[<- ] (swm.230)  -- node[left=2pt]   {\ref{thm:TFisom=>SameW}} (scdc.130);
        \draw[<->] (scdc)     -- node[left=2pt]   {\ref{thm:QGR=H}}         (tfis);
        \draw[<-,dotted] (cmn.10)   -- node[above=3pt]  {\ref{question:CDC=>comain}} (scdc.170);
        \draw (swm.0)    edge[->, bend right]  node[below right]    {\ref{corr:sameW=>sameMainSp}} (smes);
        
        \draw[dashed] (cmn) -- (join);
        \draw[dashed] (smev) -- (join);
        \draw[->] (join) -- node[below left] {\ref{thm:sameEigsandVecs}} (swm);
        
        \draw[-> ] (smes.190) -- node[rotate=-20, label={140:\ref{example:3137}}]{\scalebox{1}[2]{\bf /}}  (smev.50);
        \draw[<->] (smev)     -- node[rotate=-20, label={20:\ref{cex:sameW!=sameEigVecs} \& \ref{prop:wG!=wH}}] {\scalebox{1}[2]{\bf /}}  (swm);
        \draw[-> ] (swm.320)  -- node[rotate=-20, label={20:\ref{cex:SameW=>TFisom}}] {\scalebox{1}[2]{\bf /}}   (scdc.40);
        \draw[<->] (cmn.30)   -- node[label={[rotate=45,label distance=-4pt]90:\ref{cex:sameW} \& \ref{cex:CMN!=>SameW}}] {\scalebox{1}[2]{\bf /}} (swm.180);
        \draw[-> ] (cmn.350)  -- node[rotate=-20, label={290:\ref{cex:CMN!=>SameW}}] {\scalebox{1}[1]{\bf /}}  (scdc.190);
        \end{tikzpicture}
        \caption{The hierarchy we present through our results. The symbol $\Rightarrow$ means ``implies'', and $\nRightarrow$ means ``does not imply''. The combination $\Leftrightarrow$ is short for $\Rightarrow$ and $\Leftarrow$, i.e. ``implies and is implied by'', and similarly $\nLeftrightarrow$ is short for $\nRightarrow$ and $\nLeftarrow$, i.e. ``does not imply and is not implied by''. The dashed lines which merge at the $\land$ node denote the conjunction of those two results. The dotted lines denote \cref{question:CDC=>comain}.}
        \label{fig:charaterisation}
    \end{figure}
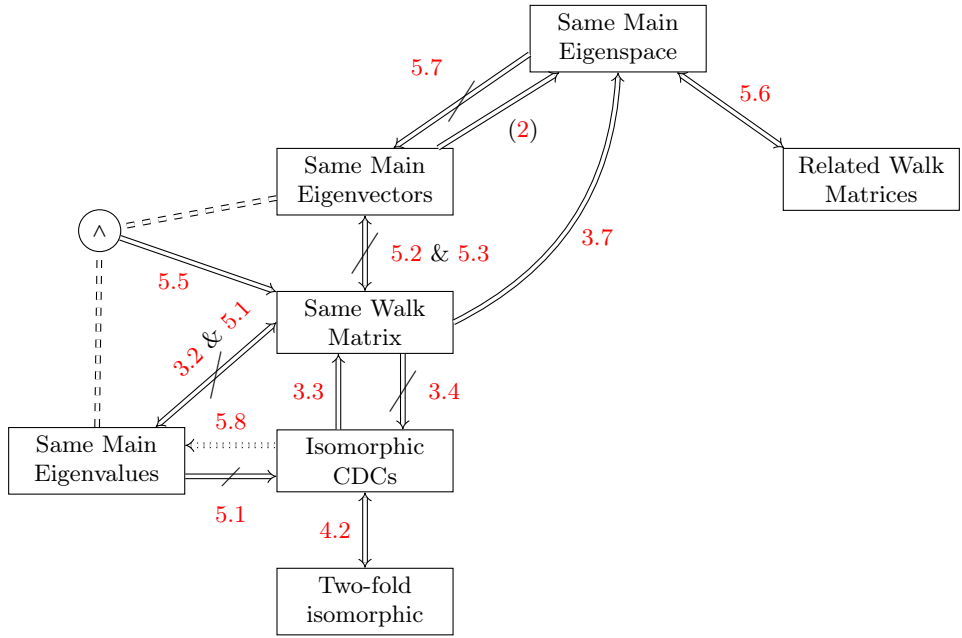

    The hierarchy of results is presented in \cref{fig:charaterisation}.  That graphs with $\CDC(G)\simeq\CDC(H)$ have the same walk matrix is established in \cref{thm:TFisom=>SameW}. That the converse of \cref{thm:TFisom=>SameW} is false, i.e., that having the same walk matrix does not imply we have isomorphic CDC's, is shown in \cref{cex:SameW=>TFisom}.   That being TF-isomorphic and having isomorphic CDC's are equivalent is established by \cref{thm:QGR=H}. 
    Now we start to fill in some of the missing links.
    
    \begin{counterexample}
        \label{cex:CMN!=>SameW}
        In \cref{cex:sameW}, two pairs of graphs are given which have the same walk matrix but different main eigenvalues.  Here we prove that the converse is also false. The graphs $G$ and $H$ of \cref{fig:counterexampleCMN=>SameW} suffice to prove that having the same main eigenvalues does not imply that the graphs have the same walk matrix.
        Indeed, they both have main polynomial $x(x^3-2x^2-4x+7)$, but their walk matrices are
        \begingroup
        \renewcommand{\arraystretch}{0.7}
        \[\mat W_{G} = \begin{pmatrix}
        1 & 2 & 6 & 12 \\
        1 & 2 & 4 & 10 \\
        1 & 2 & 4 & 10 \\
        1 & 2 & 6 & 12 \\
        1 & 4 & 8 & 24 \\
        1 & 2 & 6 & 14 \\
        1 & 2 & 6 & 14 
        \end{pmatrix}, \qquad \mat W_{H} = \begin{pmatrix}
        1 & 2 & 6 & 12 \\
        1 & 3 & 7 & 19 \\
        1 & 2 & 6 & 14 \\
        1 & 3 & 7 & 19 \\
        1 & 2 & 6 & 12 \\
        1 & 3 & 5 & 15 \\
        1 & 1 & 3 & 5
        \end{pmatrix}.\]
        \endgroup
        Moreover, their $\CDC$'s are not isomorphic.
        \begin{figure}
            \centering
            \tikzstyle{every node}=[circle, draw=black, fill=white, inner sep=0pt, minimum width=8pt, line width = 0.3mm]
            \begin{minipage}[b]{0.5\textwidth}
                \centering
                \begin{tikzpicture}[thick, scale=0.8]
                    \draw  (72:1.5)    node (1) {};
                    \draw  (144:1.5)   node (2) {};
                    \draw  (216:1.5)   node (3) {};
                    \draw  (288:1.5)   node (4) {};
                    \draw  (0:1.5)     node (5) {};
                    \draw  (3.2,1)     node (6) {};
                    \draw  (3.2,-1)    node (7) {};
                    
                    \draw[line width = 0.3mm]{
                        (1) -- (2) -- (3) -- (4) -- (5) -- (1)
                        (5) -- (6) -- (7) -- (5)
                    };
                \end{tikzpicture}
                \caption*{Graph $G$}
            \end{minipage}
            \begin{minipage}[b]{0.49\textwidth}
                \centering
                \begin{tikzpicture}[thick, scale=0.8]
                    \draw  (60:1.5)   node (1) {};
                    \draw  (120:1.5)  node (2) {};
                    \draw  (180:1.5)  node (3) {};
                    \draw  (240:1.5)  node (4) {};
                    \draw  (300:1.5)  node (5) {};
                    \draw  (0:1.5)    node (6) {};
                    \draw  (0:3.5)    node (7) {};
                    
                    \draw[line width = 0.3mm]{
                        (1) -- (2) -- (3) -- (4) -- (5) -- (6) -- (1)
                        (6) -- (7)
                        (4) -- (2)
                    };
                \end{tikzpicture}
                \vfil
                \caption*{Graph $H$}
            \end{minipage}
            \caption{Graphs $G$ and $H$ have the same main eigenvalues, but have different walk matrices and different $\CDC$'s.}
            \label{fig:counterexampleCMN=>SameW}
        \end{figure}
    \end{counterexample}

    In \cref{corr:sameW=>sameMainSp}, we show that having the same walk matrix implies that the main eigenspace is the same. But does this mean the principal main eigenvectors $\{\mat P_1\mat j,\dots,\mat P_p\mat j\}$ which generate the space are the same?
    
    \begin{counterexample}
        \label{cex:sameW!=sameEigVecs}
        Here we show that this is not the case, i.e., graphs having the same walk matrix do not necessarily have the same principal main eigenvectors. Indeed, the two pairs of graphs in \cref{cex:sameW} (\cref{fig:sameWcex}) have the same walk matrix but different principal main eigenvectors.
        
        The graphs of the first pair $(G_{5\,622},G_{12\,058})$ each have the following two corresponding principal main eigenvectors:
        \begin{align*}
            G_{5\,622}: \quad &\tfrac18(-1\pm\sqrt{65},-1\pm\sqrt{65},-1\pm\sqrt{65},-1\pm\sqrt{65},8,8,8,8)\\
            G_{12\,058}:\quad & \tfrac16(-1\pm\sqrt{37},-1\pm\sqrt{37},-1\pm\sqrt{37},-1\pm\sqrt{37},6,6,6,6)
        \end{align*}
        Clearly the vectors corresponding to $G_{5\,622}$ are not scalar multiples of those corresponding to $G_{12\,058}$, but both separately span the same main eigenspace.
    \end{counterexample}

    Thus the leap from eigenvectors to eigenspace is crucial. In fact, it turns out that if two graphs have the same main eigenvectors but different main eigenvalues, they can never have the same walk matrix:
    
    \begin{prop}
        \label{prop:wG!=wH}
        Let $G$ and $H$ be two graphs with the same main eigenvectors but different main eigenvalues. Then $\mat W_G(k) \neq \mat W_H(k)$ for all $k\geq 2$.
    \end{prop}
    
    \begin{proof}
        Let $G$ and $H$ have the same principal main eigenvectors $\{\mat x_1,\dots,\mat x_p\}$, but different eigenvalues, $\mu_1^G,\dots,\mu_p^G$,and  $\mu_1^H,\dots,\mu_p^H$ respectively. Since they are projections onto distinct eigenspaces, they are orthogonal, linearly independent, and $\mat x_1+\cdots+\mat x_p=\mat j$. Hence the second column of $\mat W_G(k)$ is
        \[\mat A_G\mat j = \sum_{i=1}^p\mat A_G \mat x_i = \sum_{i=1}^p\mu_i^G \mat x_i \neq \sum_{i=1}^p\mu_i^H \mat x_i = \sum_{i=1}^p\mat A_H \mat x_i = \mat A_H\mat j\]
        since the $\mat x_i$ are linearly independent, as required.
    \end{proof}

    \begin{example}
        \label{eg:differentWsameVecs}
        \Cref{prop:wG!=wH} establishes a non-implication. However, even though it is proven in general, we must ensure that it is not vacuously true.
        
        The graphs $G$ and $H$ in \cref{fig:differentWalkMatrices} have the same principal main eigenvectors 
        \[(\tfrac{1}{2}(-1\pm\sqrt 5), \tfrac{1}{2}(-1\pm\sqrt 5), \tfrac{1}{2}(-1\pm\sqrt 5), \tfrac{1}{2}(-1\pm\sqrt 5), 1, 1, 1, 1), \]
        but their walk matrices are 
        \begingroup
        \renewcommand{\arraystretch}{0.7}
        \[\mat W_G = \begin{pmatrix}
        1 & 2 \\
        1 & 2 \\
        1 & 2 \\
        1 & 2 \\
        1 & 4 \\
        1 & 4 \\
        1 & 4 \\
        1 & 4 
        \end{pmatrix}\qquad \text{and} \qquad \mat W_H = \begin{pmatrix}
        1 & 3 \\
        1 & 3 \\
        1 & 3 \\
        1 & 3 \\
        1 & 6 \\
        1 & 6 \\
        1 & 6 \\
        1 & 6 \\
        \end{pmatrix}. \]
        \endgroup
        Indeed, their main eigenvalues are different. The graph $G$ has main eigenvalues $1\pm \sqrt 5$, whereas $H$ has main eigenvalues $\frac{3}{2}(1\pm\sqrt5)$.
        \begin{figure}
            \centering
            \tikzstyle{every node}=[circle, draw=black, fill=white, inner sep=0pt, minimum width=8pt, line width = 0.3mm]
            \begin{minipage}[b]{0.5\textwidth}
                \centering
                \begin{tikzpicture}[thick, scale=0.8]
                    \draw  (0,3)      node (1) {};
                    \draw  (3,0)      node (2) {};
                    \draw  (0,-3)     node (3) {};
                    \draw  (-3,0)     node (4) {};
                    \draw  (-1,0.75)  node (5) {};
                    \draw  (1,0.75)   node (6) {};
                    \draw  (1,-0.75)  node (7) {};
                    \draw  (-1,-0.75) node (8) {};
                    
                    \draw[line width = 0.3mm]{
                        (1) -- (5)
                        (1) -- (6)
                        (2) -- (6)
                        (2) -- (7)
                        (3) -- (7)
                        (3) -- (8)
                        (4) -- (5)
                        (4) -- (8)
                        (5) -- (7)
                        (5) -- (8)
                        (6) -- (7)
                        (6) -- (8)
                    };
                \end{tikzpicture}
                \caption*{Graph $G$}
            \end{minipage}
            \begin{minipage}[b]{0.49\textwidth}
                \centering
                \begin{tikzpicture}[thick, scale=0.8]
                    \draw  (-2.5,2.5)   node (1) {};
                    \draw  (2.5,2.5)    node (3) {};
                    \draw  (2.5,-2.5)   node (2) {};
                    \draw  (-2.5,-2.5)  node (4) {};
                    \draw  (-1,1)       node (5) {};
                    \draw  (1,1)        node (6) {};
                    \draw  (1,-1)       node (7) {};
                    \draw  (-1,-1)      node (8) {};
                    
                    \draw[line width = 0.3mm]{
                        (1) -- (5) 
                        (1) -- (6) 
                        (1) -- (8)
                        (3) -- (5) 
                        (3) -- (6) 
                        (3) -- (7)
                        (4) -- (5) 
                        (4) -- (8) 
                        (4) -- (7)
                        (2) -- (6) 
                        (2) -- (8) 
                        (2) -- (7)
                        (5) -- (6) 
                        (5) -- (8) 
                        (5) -- (7)
                        (6) -- (8) 
                        (6) -- (7) 
                        (8) -- (7)
                    };
                \end{tikzpicture}
                \vfil
                \caption*{Graph $H$}
            \end{minipage}
            \caption{Graphs $G$ and $H$ have the same principal main eigenvectors, but have different walk matrices.}
            \label{fig:differentWalkMatrices}
        \end{figure}
    \end{example}
    
    On the other hand, the same principal main eigenvalues and eigenvectors yield the same $k$-walk matrix for any $k$, and the proof is identical:
    
    \begin{theorem}
        \label{thm:sameEigsandVecs}
        Let $k \in \Nats$, and suppose $G$ and $H$ are two comain graphs with the same principal main eigenvectors. Then \[\mat W_G(k) = \mat W_H(k).\]
    \end{theorem}
    
    \begin{proof}
        Suppose $G$ and $H$ have main eigenvalues $\mu_1,\dots,\mu_p$, and corresponding principal main eigenvectors $\mat x_1\dots,\mat x_p$. We can write $\mat j$ as $\mat x_1+\cdots+\mat x_p$.
        Now the $\ell$th column of $\mat W_G(k)$ is the vector $\mat A_G^{\ell-1}\mat j$, so
        \[\mat A_G^{\ell-1}\mat j = \sum_{i=1}^p\mat A_G^{\ell-1}\mat x_i = \sum_{i=1}^p\mu_i^{\ell-1}\mat x_i =\sum_{i=1}^p\mat A_H^{\ell-1}\mat x_i = \mat A_H^{\ell-1}\mat j,\]
        i.e., the $\ell$th column of $\mat W_H(k)$.
    \end{proof}
    
    Finally we elaborate on what is meant by ``related walk matrices'' in \cref{fig:charaterisation}. 
    
    \begin{theorem}
        \label{thm:relatedWalkMatrices}
        Let $G$ and $H$ be two graphs. Then        
        $\Main(G) = \Main(H)$ if and only if there is an invertible matrix $\mat Q$ such that $\mat W_G\mat Q = \mat W_H$.
    \end{theorem}
    
    \begin{proof}
        If $\Main(G) = \Main(H)$, then the column vectors of $\mat W_G$ and $\mat W_H$ form  bases for the same space by \cref{thm:colsOfWareBasis}. In particular, the columns of $\mat W_H$ can be expressed as a linear combination of those of $\mat W_G$. Indeed, if the $i$th column $\mat c_i$ is $\alpha_{i1}\mat j + \alpha_{i2}\mat A_G\mat j + \cdots + \alpha_{ip}{\mat A_G}^{p-1}\mat j$, then
        \begin{align*}
        \mat W_H = \begin{pmatrix}
        |&|&&|\\
        \mat c_1 & \mat c_2 & \cdots & \mat c_p\\
        |&|&&|
        \end{pmatrix} = \begin{pmatrix}
        |&|&&|\\
        \mat j &\mat A_G \mat j & \cdots & {\mat A_G}^{p-1}\mat j\\
        |&|&&|
        \end{pmatrix}
        \underbrace{\begin{pmatrix}
            \alpha_{11} & \cdots & \alpha_{1p}\\
            \vdots & \ddots & \vdots\\
            \alpha_{p1} & \cdots & \alpha_{pp}
            \end{pmatrix}}_{=\mat Q}.
        \end{align*}     
        $\mat Q$ must be invertible, since otherwise $\rank(\mat W_H)\neq p$.   
        
        Now for the converse, in $\mat W_H = \mat W_G\mat Q$ the column vectors of $\mat W_G$ are combined linearly by $\mat Q$ so they are still members of $\Main(G)$. Since $\mat Q$ is invertible, none of the columns of $\mat W_G$ become linearly dependent, so they still span all of $\Main(G)$. Thus $\Main(H) = \Main(G)$.
    \end{proof}
    
    \begin{example}
        \label{example:3137}
        An example of graphs having related walk matrices is given in \cref{fig:3137}. These correspond to graphs 31 and 37 from \plaincite{6vertices}, and were pointed out by Jeremy Curmi.\cite{Jeremy}

        \begin{figure}
            \centering
            \tikzstyle{every node}=[circle, draw=black, fill=white, inner sep=0pt, minimum width=8pt, line width = 0.3mm]
            \begin{minipage}[b]{0.5\textwidth}
                \centering
                \begin{tikzpicture}[thick, scale=0.8]
                \draw  (-2,1.5)    node (1) {};
                \draw  (-2,-1.5)   node (2) {};
                \draw  (2,1.5)     node (3) {};
                \draw  (2,-1.5)    node (4) {};
                \draw  (0,0.75)    node (5) {};
                \draw  (0,-0.75)   node (6) {};
                
                \draw[line width = 0.3mm]{
                    (1) -- (5)
                    (1) -- (6)
                    (2) -- (5)
                    (2) -- (6)
                    (3) -- (5)
                    (3) -- (6)
                    (4) -- (5)
                    (4) -- (6)
                    (5) -- (6)
                };
                \end{tikzpicture}
                \caption*{Graph $G$}
            \end{minipage}
            \begin{minipage}[b]{0.49\textwidth}
                \centering
                \begin{tikzpicture}[thick, scale=0.8]
                \draw  (-2,1.5)    node (1) {};
                \draw  (-2,-1.5)   node (2) {};
                \draw  (2,1.5)     node (3) {};
                \draw  (2,-1.5)    node (4) {};
                \draw  (0,0.75)    node (5) {};
                \draw  (0,-0.75)   node (6) {};
                
                \draw[line width = 0.3mm]{
                    (1) -- (5)
                    (1) -- (6)
                    (2) -- (5)
                    (2) -- (6)
                    (3) -- (5)
                    (3) -- (6)
                    (4) -- (5)
                    (4) -- (6)
                    (1) -- (2)
                    (3) -- (4)
                };
                \end{tikzpicture}
                \vfil
                \caption*{Graph $H$}
            \end{minipage}
            \caption{Graphs $G$ and $H$ have related walk matrices.}
            \label{fig:3137}
        \end{figure}
        Indeed, we have
        \begingroup
        \renewcommand{\arraystretch}{0.7}
        \[\mat W_G = \begin{pmatrix}
        1 & 2 \\
        1 & 2 \\
        1 & 2 \\
        1 & 2 \\
        1 & 5 \\
        1 & 5 
        \end{pmatrix} = 
        \begin{pmatrix}
        1 & 3 \\
        1 & 3 \\
        1 & 3 \\
        1 & 3 \\
        1 & 4 \\
        1 & 4 \\
        \end{pmatrix}\begin{pmatrix}
        1&-7\\
        0&3
        \end{pmatrix} = \mat W_H \begin{pmatrix}
        1&-7\\
        0&3
        \end{pmatrix} = \mat W_H\mat Q.\]
        \endgroup
        This same pair of graphs also serves as a counterexample to the following: having the same main eigenspace does not necessarily mean that they have the same main eigenvectors. Indeed, the principal main eigenvectors of $G$ are 
        $(1,1,1,1,\tfrac{1}{4}(1\pm\sqrt {33}),\tfrac{1}{4}(1\pm\sqrt {33})),$
        whereas those of $H$ are 
        $(1,1,1,1,\tfrac{1}{4}(-1\pm\sqrt {33}),\tfrac{1}{4}(-1\pm\sqrt {33})).$
    \end{example}
    
    The graphs in \cref{eg:differentWsameVecs} (\cref{fig:differentWalkMatrices}) also have related walk matrices: $\mat W_G = \left(\begin{smallmatrix}
    1&0\\
    0&\sfrac23
    \end{smallmatrix}\right) \mat W_H$.
    
    We end with a question which, if has a positive answer, would link CDC's more intimately to their main eigenvalues.
    
    \begin{question}
        \label{question:CDC=>comain}
        Let $G$ and $H$ be two graphs with $\CDC(G)\simeq \CDC(H)$. Do $G$ and $H$ have the same main eigenvalues?
    \end{question}
    
    \subsection{Finding Graphs with the same CDC}
    A simple C program was written which made use of the list of non-isomorphic graphs on 8 vertices available on Brendan McKay's website.\cite{McKay} First, the large search space of $\binom{12\,346}{2}=76\,205\,685$ pairs of non-isomorphic graphs was significantly reduced to 1\,595 pairs of graphs which are comain using the QR algorithm. This was the most intensive step computationally---it took an ordinary Linux home desktop around 25 minutes. Then another program simply found the $\CDC$'s of each of the graphs which remained, and these were compared pairwise to check for isomorphism. This took around 5 seconds, and produced 32 pairs of non-isomorphic graphs. These included all pairs on less than 8 vertices implicitly, because such pairs appear with isolated vertices added to both (by \cref{lemma:NoIsolatedVertices}).
    
    Even though the algorithm we constructed narrows the search space to consider only graphs which are comain, the list is still exhaustive; because it was determined by brute force that there are no counterexamples to the conjecture implied by \cref{question:CDC=>comain} on at most 8 vertices.

    \nocite{*}
    \bibliographystyle{unsrt}
    \bibliography{references}

\end{document}